\documentclass{amsart}

\usepackage{graphicx}
\usepackage{bm}
\usepackage{color}
\usepackage{natbib}
\usepackage{amssymb, amsmath, amsthm}
\usepackage{bm}
\theoremstyle{plain}
\usepackage{graphicx}
\usepackage[aboveskip=-5pt,position=bottom]{caption}
\usepackage{url}
\usepackage[breaklinks]{hyperref}

\let\originalleft\left
\let\originalright\right
\renewcommand{\left}{\mathopen{}\mathclose\bgroup\originalleft}
\renewcommand{\right}{\aftergroup\egroup\originalright}

\newtheorem{lemma}{Lemma}[section]

\newtheorem{theorem}[lemma]{Theorem}
\newtheorem{cor}[lemma]{Corollary}
\newtheorem{prop}[lemma]{Proposition}
\newtheorem{exam}[lemma]{\normalfont \scshape
 Example}
\newtheorem{rem}[lemma]{\normalfont \scshape Remark}
\newtheorem{defi}[lemma]{Definition}

\newcommand{\R}{\mathbb{R}}
\newcommand{\N}{\mathbb{N}}
\newcommand{\norm}[1]{\left\Vert#1\right\Vert}
\newcommand{\abs}[1]{\left\vert#1\right\vert}
\newcommand{\set}[1]{\left\{#1\right\}}

\newcommand{\eps}{\varepsilon}

\newcommand{\bfx}{\bm{x}}
\newcommand{\bfzero}{\bm{0}}

\newcommand{\bfU}{\bm{U}}

\newcommand{\bfV}{\bm{V}}

\newcommand{\bfX}{\bm{X}}
\newcommand{\bfY}{\bm{Y}}

\newcommand{\bfZ}{\bm{Z}}

\newcommand{\bfeta}{\bm{\eta}}

\newcommand{\bfxi}{\bm{\xi}}

\newcommand{\barE}{ E^-[0,1]}
\newcommand{\barC}{ C^-[0,1]}

\DeclareMathOperator{\Ei}{Ei}

\allowdisplaybreaks[4]
\begin{document}

\title[Max-Stable Processes Revisited]{Max-Stable Processes and the Functional $D$-Norm Revisited}
\author{Stefan Aulbach, Michael Falk, Martin Hofmann \and Maximilian Zott}
\address{$^1$University of W\"{u}rzburg,
Institute of Mathematics,  Emil-Fischer-Str. 30, 97074 W\"{u}rzburg, Germany.}
\email{stefan.aulbach@uni-wuerzburg.de,
michael.falk@uni-wuerzburg.de,\newline
hofmann.martin@mathematik.uni-wuerzburg.de, maximilian.zott@uni-wuerzburg.de}


\dedicatory{The final publication is available at Springer via
\url{http://dx.doi.org/10.1007/s10687-014-0210-0}}

\begin{abstract}
\citet{aulfaho11} introduced a max-domain of attraction approach for
extreme value theory in $C[0,1]$ based on functional distribution
functions, which is more general than the approach based on weak
convergence in \citet{dehal01}. We characterize this new approach by
decomposing a process into its univariate margins and its copula process.
In particular, those processes with a polynomial rate of convergence
towards a max-stable process are considered. Furthermore we investigate the
concept of differentiability in distribution of a max-stable processes.
\end{abstract}

\maketitle

\section{Introduction}
A stochastic process $\bfxi=(\xi_t)_{t\in[0,1]}$ on the interval $[0,1]$,
whose sample paths belong to the space $C[0,1]$ of continuous functions on
$[0,1]$, is called \emph{max-stable} (MSP), if there are norming functions
$a_n,b_n\in C[0,1]$, $a_n>0$, such that the distribution of the process
$\max_{1\le i\le n}\left(\bfxi^{(i)}-b_n\right)/a_n$ coincides with that of
$\bfxi$ for each $n\in\N$. By $\bfxi^{(1)},\bfxi^{(2)},\dots$ we denote
independent copies of $\bfxi$.

\citet{aulfaho11} established a characterization of the distribution of an
MSP via a norm on $E[0,1]$, the space of bounded functions on $[0,1]$ which
have finitely many discontinuities. This norm  is called $D$-norm and it is
defined by means of a so-called generator process.

An MSP $\bfeta=(\eta_t)_{t\in[0,1]}\in C[0,1]$ with standard negative
exponential margins $P(\eta_t\le x)=\exp(x)$, $x\le 0$, will be called a
\emph{standard max-stable} process (SMSP). As the distribution of a
stochastic process on $[0,1]$ is determined by its finite dimensional
marginal distributions, a process $\bfeta\in C[0,1]$ with identical marginal distribution function
(df) $F(x)=\exp(x)$, $x\le 0$, is standard max-stable if and only if
\begin{equation} \label{eqn:standard_max-stable}
P\left(\bfeta\le \frac fn\right)^n=P(\bfeta\le f),\qquad f\in E^-[0,1],\,n\in\N,
\end{equation}
where $E^-[0,1]$ denotes the subset of those functions in $E[0,1]$, which
attain only non positive values. Note that it would be sufficient in equation \eqref{eqn:standard_max-stable} to consider
$f\in C^-[0,1]$ of non positive continuous functions. The extension
to $E^-[0,1]$, however, provides the inclusion of the finite dimensional marginal
distributions (fidis) of $\bfeta$, as
\[
P(\eta_{t_i}\le x_i,\ 1\le i\le d) = P(\bfeta\le f)
\]
where $0\le t_1<\dots<t_d\le1$ and $f\in E^-[0,1]$ is given by $f(t_i) =
x_i<0$ for $i\in\set{1,\dots,d}$ and $f(t) = 0$ for
$t\in[0,1]\setminus\set{t_1,\dots,t_d}$.

From \citet{aulfaho11} we know that \eqref{eqn:standard_max-stable} is
equivalent with the condition that there is some norm $\norm\cdot_D$ on
$E[0,1]$, called \emph{$D$-norm}, satisfying
\begin{equation}\label{eqn:df_of_smsp}
P(\bfeta\le f)=\exp\left(-\norm f_D\right),\qquad f\in E^-[0,1].
\end{equation}
Precisely, there exists a stochastic process $\bfZ=(Z_t)_{t\in[0,1]}\in
C[0,1]$ with
\[
0\le Z_t\le m,\quad E(Z_t)=1,\qquad t\in[0,1],
\]
for some number $m\ge 1$, such that
\[
\norm f_D=E\left(\sup_{t\in[0,1]}(\abs{f(t)}Z_t)\right),\qquad f\in E[0,1].
\]
The condition $Z_t\le m$, $t\in[0,1]$, can be weakened to
$E\left(\sup_{t\in[0,1]}Z_t\right)<\infty$. Observe that property
\eqref{eqn:df_of_smsp} of a max-stable process $\bfeta$ with standard negative exponential margins corresponds to the spectral representation of a max-stable process with unit Fr\'{e}chet margins as considered in \citet{resroy91} and \citet{dehaan84}, since
$P\left(-\eta_t^{-1}\le y\right) = \exp\left(-y^{-1}\right)$ for $y>0$ and
$t\in[0,1]$.

Based on this characterization, \citet{aulfaho11} introduced a functional
domain of attraction approach for stochastic processes in terms of
convergence of their \emph{distribution functions}. This functional domain of attraction is larger
than the one based on \emph{weak convergence} as investigated in
\citet{dehal01}, cf. \citet[Proposition\,5]{aulfaho11}. In Section
\ref{sec:characterization_of_fmda} of the present paper we will carry over
\citeauthor{dehal01}'s (\citeyear{dehal01}) characterization of max-domain of
attraction for stochastic processes in $C[0,1]$ in terms of weak convergence
to our domain of attraction approach based on convergence of df.

\citet{buihz08} suggested the definition of \emph{generalized Pareto
processes} (GPP), which extends the multivariate Peaks-Over-Threshold approach to function spaces.
This particular approach was investigated and settled in \citet{ferrdh12},
\citet{aulfaho11} and \citet{domri13}. In Section
\ref{sec:Spectral_delta_Neighborhood_of_a_standard_GPP} we will introduce certain
$\delta$-neighborhoods of GPP, which can be characterized by their rate of
convergence towards a max-stable process. This is in complete accordance with
the multivariate case.

Finally, we establish the concept of differentiability in distribution of an SMSP in Section 4. To this end, we investigate some properties of SMSP such as the partial derivatives of a $D$-norm, the distribution of the increments of an SMSP and the conditional distribution of an SMSP given one point being observed.

To improve the readability of this paper we use bold face such as $\bfxi$,
$\bfX$ for stochastic processes and default font $f$, $a_n$ etc. for non
stochastic functions. Operations on functions such as $\bfxi<f$ or
$(\bfxi-b_n)/a_n$ are meant pointwise. The usual abbreviations \textit{iid, a.s.} and \textit{rv} for the terms \textit{independent and identically distributed, almost surely} and \textit{random
variable}, respectively,  are used.

\section{A Characterization of Max-Domain of Attraction}\label{sec:characterization_of_fmda}

In the multivariate framework, it is well-known that a rv
$(X_1,\dots,X_d)$ is in the domain of attraction of a multivariate max-stable
distribution if and only if its copula has this property and the distribution
of $X_i$ is  in the univariate domain of attraction of a
max-stable distribution for each $i=1,\dots,n$. We refer to \citet{gal78}, \citet{deheu78,deheu84}
and \citet{aulbf11} for details.

\Citet{dehal01} extended this result to the usual weak convergence of the maximum of $n$ iid stochastic processes, linearly standarized, towards a max-stable process in $C[0,1]$. They provided a condition which entails a characterization of the \emph{domain of attraction} property in the sense of usual weak convergence in $C[0,1]$ of stochastic processes. This condition is essentially condition~\eqref{eqn:final_equivalent_crucial_condition_on_marginal_df} below. The characterization of the domain of attraction property then consists of uniform weak convergence on $[0,1]$ of the \emph{marginal} maxima to a univariate extreme value distribution, together with weak convergence of the corresponding \emph{copula process} in function space. We will carry over
\citeauthor{dehal01}'s (\citeyear{dehal01}) result to our domain of
attraction approach based on convergence of dfs of stochastic processes, see Theorem \ref{thm:equivalence_of_convergence} and Corollary \ref{cor:equivalence_of_fda} below. As established in \citet[Proposition 5]{aulfaho11}, the max-domain of attraction we consider is larger than that based on ordinary weak convergence. But it is still an open problem whether it is \emph{strictly} larger.

Let $\bfX=(X_t)_{t\in[0,1]}\in C[0,1]$ be a
stochastic process with continuous marginal df $F_t(x)=P(X_t\le x)$,
$x\in\R$, $t\in[0,1]$, and let $\bfxi=(\xi_t)_{t\in[0,1]}\in C[0,1]$ be an
MSP with marginal df $G_t$, $t\in[0,1]$. Suppose that there exist norming
functions $a_n, b_n\in C[0,1]$, $a_n>0$, $n\in\N$, such that
\begin{equation}\label{eqn:final_equivalent_crucial_condition_on_marginal_df}
\sup_{t\in[0,1]}\abs{n\big(F_t(a_n(t)f(t)+b_n(t)) - 1\big)-\log\big(G_t(f(t))\big)}\to_{n\to\infty} 0
\end{equation}
for each $f\in E[0,1]$ with $\inf_{t\in[0,1]}G_t(f(t))>0$.
This is essentially condition (3.11) in
\citet{dehal01}.  Using Taylor expansion
$\log(1+\varepsilon)=\varepsilon + O(\varepsilon^2)$ as $\varepsilon\to 0$,
condition \eqref{eqn:final_equivalent_crucial_condition_on_marginal_df} in particular implies weak convergence of the univariate margins
\[
F_t(a_n(t)x+b_n(t))^n\to_{n\to\infty} G_t(x),\qquad x\in\R, t\in[0,1].
\]

Put $\bfU:=(U_t)_{t\in[0,1]}:=(F_t(X_t))_{t\in[0,1]}$, which is the
\emph{copula process} corresponding to $\bfX$. Let
$\bfU^{(1)},\bfU^{(2)},\dots$ be independent copies of $\bfU$, and let
$\bfX^{(1)},\bfX^{(2)},\dots$ be independent copies of $\bfX$. The following
theorem is the main result of this section.

For the implication
$\eqref{eqn:functional_domain_of_attraction_assumption}\implies
\eqref{eqn:copula_domain_of_attraction_assumption}$ we set
$\eta_t:=\log(G_t(\xi_t))$, $t\in[0,1]$, and for the reverse conclusion
$\xi_t:=G_t^{-1}(\exp(\eta_t))$, $t\in [0,1]$. In both cases the processes
$\bfxi:=(\xi_t)_{t\in[0,1]}$ and $\bfeta:=(\eta_t)_{t\in[0,1]}\in C[0,1]$ are
max-stable, $\bfeta$ being an SMSP. By Lemma 1 in \citet{aulfaho11} or the elementary arguments as in the proof
of Theorem 9.4.1 in \citet{dehaf06}, one obtains that
$P(G_t(\xi_t)=0\mathrm{\ for\ some\ }t\in[0,1])=0$, i.e., the processes
$\bfeta$ and $\bfxi$ are well defined.

\begin{theorem}\label{thm:equivalence_of_convergence}
We have under condition \eqref{eqn:final_equivalent_crucial_condition_on_marginal_df}
\begin{equation}\label{eqn:functional_domain_of_attraction_assumption}
P\left(\max_{1\le i\le n}\frac{\bfX^{(i)}-b_n}{a_n}\le f\right) \to_{n\to\infty} P(\bfxi\le f), \qquad f\in E[0,1],
\end{equation}
if and only if
\begin{equation}\label{eqn:copula_domain_of_attraction_assumption}
P\left(n\max_{1\le i\le n}\left(\bfU^{(i)}-1\right)\le g\right) \to_{n\to\infty} P(\bfeta\le g), \qquad g\in E^-[0,1].
\end{equation}
\end{theorem}

\begin{proof}
As $\bfxi$ has continuous sample paths and $G_t$ is continuous for each
$t\in[0,1]$, we have continuity of the function $[0,1]\ni t\mapsto G_t(x)$
uniformly for $x\in\R$. Thus
\begin{equation*}
  \abs{G_t(x)-G_{t_0}(x_0)}
  \le \abs{G_t(x)-G_{t_0}(x)} + \abs{G_{t_0}(x)-G_{t_0}(x_0)}
\end{equation*}
implies the continuity of the function $[0,1]\times\R\ni (t,x)\mapsto
G_t(x)$. This shows in particular that any discontinuity of $[0,1]\ni
t\mapsto G_t(h(t))$ is also a discontinuity of $h\in E[0,1]$. The same
arguments apply to the function $[0,1]\times\R\ni (t,x)\mapsto F_t(x)$.

We first establish the implication
$\eqref{eqn:functional_domain_of_attraction_assumption}\implies
\eqref{eqn:copula_domain_of_attraction_assumption}$. Choose $g\in E^-[0,1]$
with $\sup_{t\in[0,1]}g(t)<0$ and put $f(t):=G_t^{-1}(\exp(g(t))$. Then $f\in
E[0,1]$ and we obtain from assumption
\eqref{eqn:functional_domain_of_attraction_assumption}
\begin{equation}\label{eqn:variant_of_fda_assumption}
  P\left(\max_{1\le i\le n}\bfX^{(i)}\le a_n f+b_n\right)\to_{n\to\infty}P(\bfxi\le f)=P(\bfeta\le g)=\exp\left(-\norm g_D\right),
\end{equation}
where $\norm\cdot_D$ is the $D$-norm corresponding to the SMSP $\bfeta$.

We have, on the other hand, by condition
\eqref{eqn:final_equivalent_crucial_condition_on_marginal_df} and the continuity of the marginal df $F_t$
\begin{align*}
  &P\left(\max_{1\le i\le n}\bfX^{(i)}\le a_n f+b_n\right)\\
  &= P\left(n\max_{1\le i\le n}\left(U_t^{(i)}-1\right)\le  n\big(F_t(a_n(t) f(t)+b_n(t))-1\big),\,t\in[0,1]\right)\\
  &= P\left(n\max_{1\le i\le n}\left(U_t^{(i)}-1\right)\le g(t) + r_n(t),\,t\in[0,1]\right),
\end{align*}
where $r_n(t)=o(1)$ as $n\to\infty$, uniformly for $t\in[0,1]$. As an intermediate step we claim that
\begin{equation}\label{eqn:fda_of_copula process_for_g_<_0}
P\Bigl(n\max_{1\le i\le n}\left(\bfU^{(i)}-1\right)\le g\Bigr) \to_{n\to\infty} P(\bfeta\le g),\quad g\in E^-[0,1],\,\sup_{t\in[0,1]}g(t)<0.
\end{equation}
Replace $g$ by $g+\varepsilon$ and $g-\varepsilon$ for $\varepsilon>0$ small
enough such that $g+\varepsilon<0$, and put
\[
f_\varepsilon(t):= G_t^{-1}(\exp(g(t)+\varepsilon)),\quad f_{-\varepsilon}(t):= G_t^{-1}(\exp(g(t)-\varepsilon)),\qquad t\in[0,1].
\]
Then $f_\varepsilon, f_{-\varepsilon}\in E[0,1]$, and we obtain from
condition \eqref{eqn:final_equivalent_crucial_condition_on_marginal_df} and
equation \eqref{eqn:variant_of_fda_assumption} for $n\ge n_0$
\begin{align*}
&P\left(n\max_{1\le i\le n}\left(U_t^{(i)}-1\right)\le n\big( F_t(a_n(t)f_\varepsilon(t)+b_n(t))-1\big),\,t\in[0,1]\right)\\
&\ge P\left(n\max_{1\le i\le n}\left(\bfU^{(i)}-1\right)\le g\right)\\
&\ge P\left(n\max_{1\le i\le n}\left(U_t^{(i)}-1\right)\le n\big( F_t(a_n(t)f_{-\varepsilon}(t)+b_n(t))-1\big),\,t\in[0,1]\right),
\end{align*}
where the upper bound converges to $\exp\left(-\norm{g+\varepsilon}_D\right)$
and the lower bound to $\exp\left(-\norm{g-\varepsilon}_D\right)$. As both
converge to $\exp\left(-\norm g_D\right)$ as $\varepsilon\to 0$, we have
established \eqref{eqn:fda_of_copula process_for_g_<_0}.

Next we claim that \eqref{eqn:fda_of_copula process_for_g_<_0} is true for
each $g\in E^-[0,1]$, i.e., we drop the assumption $\sup_{t\in[0,1]}g(t)<0$.
From \eqref{eqn:fda_of_copula process_for_g_<_0} we
deduce that for each $\varepsilon>0$
\[
\lim_{n\to\infty} P\left(n\max_{1\le i\le n}\left(\bfU^{(i)}-1\right)\le g-\varepsilon\right)=\exp\left(-\norm{g-\varepsilon}_D\right)
\]
and, thus,
\begin{align*}
\liminf_{n\to\infty} P\left(n\max_{1\le i\le n}\left(\bfU^{(i)}-1\right)\le g\right)
&\ge \lim_{n\to\infty} P\left(n\max_{1\le i\le n}\left(\bfU^{(i)}-1\right)\le g-\varepsilon\right)\\
&= \exp\left(-\norm{g-\varepsilon}_D\right).
\end{align*}
As $\varepsilon > 0$ was arbitrary, we have established that
\[
\liminf_{n\to\infty} P\left(n\max_{1\le i\le n}\left(\bfU^{(i)}-1\right)\le g\right)\ge \exp\left(-\norm g_D\right).
\]

On the other hand, we have by \eqref{eqn:fda_of_copula process_for_g_<_0}
for $\varepsilon>0$
\[
\lim_{n\to\infty} P\left(n\max_{1\le i\le n}\left(\bfU^{(i)}-1\right)\le -\varepsilon\right) =\exp\left(-\varepsilon\norm 1_D\right)\to_{\varepsilon\downarrow 0} 1,
\]
and, thus,
\begin{align*}
&\limsup_{n\to\infty} P\left(n\max_{1\le i\le n}\left(\bfU^{(i)}-1\right)\le g\right)\\
&\le \limsup_{n\to\infty}\Biggl( P\left(n\max_{1\le i\le n}\left(\bfU^{(i)}-1\right)\le g,\, n\max_{1\le i\le n}\left(\bfU^{(i)}-1\right)\le -\varepsilon\right)\\
&\hspace*{5cm}+ P\left(\left(n \max_{1\le i\le n}\left(\bfU^{(i)}-1\right) \le -\varepsilon\right)^\complement\right)\Biggr)\\
&= \exp\left(-\norm{(\min(g(t),-\varepsilon)_{t\in[0,1]}}_D\right) + 1-\exp\left(-\varepsilon \norm 1_D\right)
\end{align*}
by \eqref{eqn:fda_of_copula process_for_g_<_0}. As the first term in the
final line above converges to $\exp\left(-\norm g_D\right)$ as
$\varepsilon\downarrow 0$ and the second one to zero, we have established that
\[
\limsup_{n\to\infty} P\left(n\max_{1\le i\le n}\left(\bfU^{(i)}-1\right)\le g\right)\le \exp\left(-\norm g_D\right).
\]
This proves equation \eqref{eqn:fda_of_copula process_for_g_<_0} for
arbitrary $g\in E^-[0,1]$ and completes the proof of the conclusion
$\eqref{eqn:functional_domain_of_attraction_assumption}\implies
\eqref{eqn:copula_domain_of_attraction_assumption}$.

Next we establish the implication
$\eqref{eqn:copula_domain_of_attraction_assumption}\implies
\eqref{eqn:functional_domain_of_attraction_assumption}$. Choose $f\in E[0,1]$
with $\inf_{t\in[0,1]}G_t(f(t))>0$ and put $g(t):=\log(G_t(f(t)))$,
$t\in[0,1]$. From the assumption
\eqref{eqn:copula_domain_of_attraction_assumption} we obtain
\[
P\left(n\max_{1\le i\le n}\left(\bfU^{(i)}-1\right)\le g\right)\to_{n\to\infty} P(\bfeta\le g)=P(\bfxi\le f)=\exp\left(-\norm g_D\right).
\]

On the other hand, we have by condition
\eqref{eqn:final_equivalent_crucial_condition_on_marginal_df}
\begin{align*}
&P\left(n\max_{1\le i\le n}\left(\bfU^{(i)}-1\right)\le g\right)\\
&= P\left(n\max_{1\le i\le n}\left(\bfU_t^{(i)}-1\right)\le n\big(F_t(a_n(t) f(t)+b_n(t))-1\big) +  r_n(t),\,t\in[0,1]\right),
\end{align*}
where $r_n(t)=o(1)$ as $n\to\infty$, uniformly for $t\in[0,1]$. We claim that
\begin{align}\label{eqn:fda_of_copula_processes}
&P\left(n\max_{1\le i\le n}\left(\bfU_t^{(i)}-1\right)\le n\big(F_t(a_n(t) f(t)+b_n(t))-1\big),\,t\in[0,1]\right)\nonumber\\
&\to_{n\to\infty} P\left(\bfeta\le g\right)
=\exp\left(-\norm g_D\right).
\end{align}
Replace $g$ by $\min(g+\varepsilon,0)$ and $g-\varepsilon$, where
$\varepsilon>0$ is arbitrary. Then we obtain from
\eqref{eqn:copula_domain_of_attraction_assumption} and condition
\eqref{eqn:final_equivalent_crucial_condition_on_marginal_df} for $n\ge n_0=n_0(\varepsilon)$
\begin{align*}
  &P\left(n\max_{1\le i\le n}\left(\bfU^{(i)}-1\right)\le \min(g+\varepsilon,0)\right)\\
  &=P\left(n\max_{1\le i\le n}\left(\bfU^{(i)}-1\right)\le g+\varepsilon\right)\\
  &\ge P\left(n\max_{1\le i\le n}\left(\bfU^{(i)}-1\right)\le n\big(F_t(a_n(t) f(t)+b_n(t))-1\big),\,t\in[0,1]\right)\\
  &\ge P\left(n\max_{1\le i\le n}\left(\bfU^{(i)}-1\right)\le g-\varepsilon\right).
\end{align*}
Since
\begin{equation*}
P\left(n\max_{1\le i\le n}\left(\bfU^{(i)}-1\right)\le \min(g+\varepsilon,0)\right)\to_{n\to\infty}\exp\left(-\norm{\min(g+\varepsilon,0)}_D\right)
\end{equation*}
as well as
\begin{equation*}
P\left(n\max_{1\le i\le n}\left(\bfU^{(i)}-1\right)\le
g-\varepsilon\right)\to_{n\to\infty}
\exp\left(-\norm{g-\varepsilon}_D\right)
\end{equation*}
and $\varepsilon>0$ was arbitrary, \eqref{eqn:fda_of_copula_processes}
follows.

From the equality
\begin{align*}
&P\left(n\max_{1\le i\le n}\left(\bfU^{(i)}-1\right)\le n\big(F_t(a_n(t) f(t)+b_n(t))-1\big),\,t\in[0,1]\right)\\
&=P\left(\max_{1\le i\le n}\bfX^{(i)}\le a_nf+b_n\right)
\end{align*}
and from \eqref{eqn:fda_of_copula_processes} we obtain
\begin{equation}\label{eqn:fda_for_X}
\lim_{n\to\infty} P\left(\max_{1\le i\le n}\bfX^{(i)}\le a_nf+b_n\right) =P(\bfxi\le f)
\end{equation}
for each $f\in E[0,1]$ with $\inf_{t\in[0,1]}G_t(f(t))>0$. If
$\inf_{t\in[0,1]}G_{t}(f(t))=0$, then, for $\varepsilon >0$, there  exists
$t_0\in[0,1]$ such that $G_{t_0}(f(t_0))\le\varepsilon$. We, thus, have
$P(\bfxi\le f)\le P\left(\xi_{t_0}\le f(t_0)\right)=G_{t_0}(f(t_0))\le
\varepsilon$ and, by condition
\eqref{eqn:final_equivalent_crucial_condition_on_marginal_df},
\begin{align*}
P\left(\max_{1\le i\le n}\bfX^{(i)}\le a_nf+b_n\right)&\le P\left(\max_{1\le i\le n} X_{t_0}^{(i)}\le a_n(t_0)f(t_0)+ b_n(t_0)\right)\\
&\to_{n\to\infty} G_{t_0}(f(t_0))
\le \varepsilon.
\end{align*}
As $\varepsilon>0$ was arbitrary, we have established
\[
\lim_{n\to\infty} P\left(\max_{1\le i\le n}\bfX^{(i)}\le a_nf+b_n\right)=0=P(\bfxi\le f)
\]
in that case, where $\inf_{t\in[0,1]}G_t(f(t))=0$ and, thus,
\eqref{eqn:fda_for_X} for each $f\in E[0,1]$. This completes the proof of
Theorem \ref{thm:equivalence_of_convergence}.
\end{proof}

We say that a stochastic process $\bfX=(X_t)_{t\in[0,1]}\in C[0,1]$ is in the
\emph{domain of attraction} of a max-stable process
$\bfxi=(\xi_t)_{t\in[0,1]}\in C[0,1]$, denoted by $\bfX\in\mathcal D(\bfxi)$,
if there are norming functions $a_n>0$, $b_n\in C[0,1]$, $n\in\N$, such that
\eqref{eqn:functional_domain_of_attraction_assumption} holds. The following
consequence of Theorem \ref{thm:equivalence_of_convergence} extends the
well-known characterization of \emph{multivariate} domain of attraction in
terms of weak convergence of the univariate margins together with weak
convergence of the copulas (\citet{gal78}, \citet{deheu78,deheu84} and
\citet{aulbf11}) to the functional space.

\begin{cor}\label{cor:equivalence_of_fda}
Let $\bfX=(X_t)_{t\in[0,1]}\in C[0,1]$ be a stochastic process with identical
continuous marginal df $F(x)=P(X_t\le x)$, $x\in\R$, $t\in[0,1]$, and let
$\bfxi=(\xi_t)_{t\in[0,1]}\in C[0,1]$ be an MSP with identical marginal df
$G$. Denote by $\bfU=(U_t)_{t\in[0,1]}:=(F(X_t))_{t\in[0,1]}$  the copula
process pertaining to $\bfX$. Then we have $\bfX\in\mathcal D(\bfxi)$ if and
only if $\bfU\in\mathcal D(\bfeta)$ and $F\in \mathcal D(G)$.
\end{cor}

For a characterization of the condition $\bfU\in \mathcal D(\bfeta)$ in terms
of certain neighborhoods of generalized Pareto processes see Proposition
\ref{prop:characterization_of_D(eta)} below.

\begin{proof}[Proof of Corollary \ref{cor:equivalence_of_fda}]
The assumption $F\in\mathcal D(G)$ yields
$\sup_{x\in\R}|F^n(a_nx+b_n)-G(x)|\to_{n\to\infty}0$ for some sequence of
norming constants $a_n>0$, $b_n\in\R$, $n\in\N$. Taking logarithms and using
Taylor expansion of $\log(1+G(x))$ for $x\in[x_0,x_1]$ with $G(x_0)>0$ implies
\[
\sup_{x\in[x_0,x_1]}\abs{n(F(a_nx+b_n)-1)-\log(G(x))}\to_{n\to\infty}0
\]
and, thus, condition \eqref{eqn:final_equivalent_crucial_condition_on_marginal_df} is
satisfied. Corollary \ref{cor:equivalence_of_fda} is now an immediate
consequence of Theorem \ref{thm:equivalence_of_convergence} together with the
fact that the assumption $\bfX\in\mathcal D(\bfxi)$ implies in particular
that $F\in\mathcal D(G)$.
\end{proof}

We conclude this section with a short remark. Choose $f\in E[0,1]$ which is not the constant function zero.  If  $\bfeta\in C^-[0,1]$ is an SMSP, then  the univariate rv
\[
\eta_f:= \sup_{t\in [0,1]} \frac{\eta_t}{\abs{f(t)}},
\]
is by equation \eqref{eqn:df_of_smsp} negative exponential distributed with parameter $\norm f_D$. This explains why $P\left(\sup_{t\in[0,1]}\eta_t<0\right)=1$.

\section{\texorpdfstring{$\delta$}{Delta}-Neighborhood of a Generalized Pareto Process} \label{sec:Spectral_delta_Neighborhood_of_a_standard_GPP}
First we recall some facts from \citet{aulfaho11}. A univariate generalized
Pareto distribution (GPD) $W$ is simply given by $W(x)=1+\log(G(x))$,
$G(x)\ge 1/e$, where $G$ is a univariate extreme value distribution (EVD). It
was established by \citet{pick75} and \citet{balh74} that, roughly, the
maximum of $n$ iid univariate observations, linearly standardized, converges
in distribution to an EVD as $n$ increases if, and only if, the exceedances
above an increasing threshold follow a generalized Pareto distribution (GPD).
The multivariate extension is due to \citet{roott06}.

Theorem 2.2.5 in \citet{fahure10} characterizes \emph{$\delta$-neighborhoods} of univariate GPD via of a \emph{polynomial rate of convergence} of the maximum in an iid sample towards its limiting extreme value distribution. The \emph{multivariate} extension is given in Theorem 5.5.5 in \citet{fahure10}. In Proposition \ref{prop:characterization_rate_of_convergence} we will extend this characterization of a polynomial rate of convergence to \emph{function space}. We, therefore, have to introduce generalized Pareto processes and their spectral $\delta$-neighborhoods first.

It was observed by
\citet{buihz08} that a $d$-dimensional GPD $W$ with ultimately standard
Pareto margins can be represented in its upper tail  as
$W(\bfx)=P(U^{-1}\bfZ\le\bfx)$, $\bfzero\le \bfx_0\le \bfx\in\R^d$, where the
rv $U$ is uniformly on $(0,1)$ distributed and independent of the rv
$\bfZ=(Z_1,\dots,Z_d)$ with $0\le Z_i\le m$ for some $m\ge 1$ and $E(Z_i)=1$,
$1\le i\le d$. The following definition extends this approach to function
spaces.

\begin{defi}\upshape
Let $U$ be a rv which is uniformly distributed on $[0,1]$ and independent of
the generator process $\bfZ\in C[0,1]$ that is characterized by the two
properties
\begin{equation*}
0\le Z_t\le m\quad\mathrm{\ and\ }\quad E(Z_t)=1,\qquad t\in[0,1],
\end{equation*}
for some constant $m\ge 1$. Then the stochastic process
\[
\bfY:=\frac 1U \bfZ\ \in C^+[0,1]
\]
is a \emph{generalized Pareto process} (GPP) with ultimately standard Pareto margins \citet{buihz08}.
\end{defi}

Corollary 9.4.5 in \citet{dehaf06} implies that one can choose the generator process $\bfZ$ such that $\sup_{t\in[0,1]}Z_t=\mathit{const}$ a.s.

Characterizations of GPP in terms of a functional Peaks-Over-Threshold approach  are established in \citet{ferrdh12} and \citet{domri13}. Local asymptotic normality in $\delta$-neighborhoods of standard generalized Pareto processes was established in \citet{aulfa11}, and tests for a generalized Pareto process are investigated in \citet{aulfa12}.

The univariate margins $Y_t$ of $\bfY$ have ultimately standard Pareto tails:
\begin{align*}
P(Y_t\le x)&= P\left(\frac 1 x Z_t\le U\right)\\
&=\int_0^m P\left(\frac 1x z\le U\right)\,(P*Z_t)(dz)\\
&=1 - \frac 1 x \int_0^m z\,(P*Z_t)(dz)\\
&=1 - \frac 1 x E(Z_t)\\
&= 1- \frac 1 x,\qquad x\ge m,\,0\le t\le 1,
\end{align*}
where $P*Z_t$ denotes the distribution of $Z_t$.

Put $\bfV:=\max(-1/\bfY, M)$, where $M<0$ is an arbitrary constant, which
ensures that $\bfV>-\infty$. Then, as shown in Example 1 in \citet{aulfaho11},
\begin{equation*}
  P\left(\bfV\le f\right)=P\left(\sup_{t\in[0,1]}\left(\abs{f(t)}Z_t\right)\le U\right)
  =1-\norm f_D
\end{equation*}
for all $f\in\barE$ with $\norm f_\infty\le \min(1/m,\abs M)$, i.e., $\bfV$
has the property that its functional df is in its upper tail equal to
\begin{equation}\label{uppertail_funct_df_of_GPD}
W(f):=P\left(\bfV\le f\right)
=1+\log(G(f)),\quad f\in \barE,\,\norm f_\infty\le \min(1/m,\abs M),
\end{equation}
where $G(f)=P(\bfeta\le f)$ is the functional df of the MSP $\bfeta$ with
$D$-norm $\norm\cdot_D$ and generator $\bfZ$. This representation of the
upper tail of a GPP in terms of $1+\log(G)$ is in complete accordance with
the uni- and multivariate case \citep[see, for example,][Chapter
5]{fahure10}. We write $W=1+\log(G)$ in short notation and call $\bfV$ a GPP
as well, this time with ultimately uniform margins.

\begin{rem}\upshape
Due to representation \eqref{uppertail_funct_df_of_GPD}, the GPP $\bfV$ is
obviously in the functional domain of attraction of an SMSP $\bfeta$ with
$D$-norm $\norm\cdot_D$ and generator $\bfZ$: Take $a_n=1/n$ and $b_n=0$. We
have for $f\in E^-[0,1]$ and large enough $n\in\N$
\[
P\left(\bfV\le \frac 1n f\right)^n =\left(1-\frac 1n \norm f_D\right)^n \to_{n\to\infty} \exp\left(-\norm f_D\right)=P(\bfeta\le f).
\]
\end{rem}

The following result is a functional version of the well-known fact  that the
spectral df of a GPD rv is equal to a uniform df in a neighborhood of zero.

\begin{lemma}
We have for $f\in\barE$  with $0<\norm f_\infty\le m$ and some $s_0<0$
\[
W_f(s):=P\left(\bfV\le s\abs f\right)=1+s\norm f_D,\qquad s_0\le s\le 0.
\]
\end{lemma}

The next result is established in
\citet{aulfaho11}. Let $\bfU$ be a copula process and let $\bfeta$ be an SMSP with functional df $P(\bfeta\le f)=\exp\left(-\norm f_D\right)$, $f\in E^-[0,1]$.

\begin{prop}\label{prop:characterization_of_D(eta)}
The property $\bfU\in \mathcal{D}(\bfeta)$ is equivalent to the condition
\begin{equation}\label{eqn:tail_equivalence_of_H_and_W}
\lim_{s\uparrow 0} \frac{1-H_f(s)}{1-W_f(s)}=1, \qquad f\in\barE,
\end{equation}
i.e., the spectral df $H_f(s)=P(\bfU-1\le s\abs f),\;s\le 0,$ of $\bfU-1$ is for each $f\in E^-[0,1]$ with $\norm f_\infty>0$
\textit{tail equivalent} to that of the GPD $W_f(s)=1+s\norm f_D$, $s_0\le s\le 0$.
\end{prop}

This characterization of the domain of attraction of an SMSP in terms of a
certain GPP suggests to focus on the following standard case. Recall that
Section~\ref{sec:characterization_of_fmda} justified to consider SMSPs in
place of general MSPs.
\begin{defi}\upshape
A stochastic process $\bfV\in \barC$ is a \emph{standard generalized Pareto
process} (SGPP), if there exists a $D$-norm $\norm\cdot_D$ on $E[0,1]$ and
some $c_0>0$ such that
\[
 P(\bfV\le f)=1-\norm f_D
\]
for all $f\in\barE$ with $\norm f_\infty\le c_0$.
\end{defi}

The same arguments as at the end of Section
\ref{sec:characterization_of_fmda} provide the upper tail of the df of the rv
\[
V_f :=  \sup_{t\in[0,1]} \frac{V_t}{\abs{f(t)}}
\]
if $\bfV\in C^-[0,1]$ is an SGPP and $f\in E[0,1]$ is not the constant
function zero. We obtain
\[
P\left(V_f>x\right)=\norm f_D \abs x,\qquad -1\le x\le 0,
\]
if $\norm f_\infty\le c_0$, i.e., $V_f$ follows in its upper tail, precisely
on $(-1,0)$, a uniform distribution.

Using the spectral decomposition of a stochastic process in $\barC$, we can
easily extend the definition of a spectral $\delta$-neighborhood of a
multivariate GPD as in \citet[Section 5.5]{fahure10} to the spectral
$\delta$-neighborhood of an SGPP. Denote by
$E^-_1[0,1]=\set{f\in\barE:\,\norm f_\infty=1}$ the unit sphere in $\barE$.

\begin{defi}\upshape
We say that a stochastic process $\bfY\in\barC$ belongs to the
\textit{spectral $\delta$-neighborhood} of the SGPP $\bfV$ for some $\delta
\in(0,1]$, if we have uniformly for $f\in E^-_1[0,1]$ the expansion
\begin{align*}
  1-P(\bfY\le c f)&=(1-P(\bfV\le cf))\left(1+O\left(c^\delta\right)\right)\\
  &=c\norm f_D \left(1+O\left(c^\delta\right)\right)
\end{align*}
as $c\downarrow 0$.
\end{defi}

An SMSP is, for example, in the spectral $\delta$-neighborhood of the
corresponding GPP with $\delta=1$.

The following result extends Theorem 5.5.5 in \citet{fahure10} on the rate of
convergence of multivariate extremes. It shows that $\delta$-neighborhoods
collect, roughly, all processes which have a polynomial rate of convergence
towards an SMSP.

\begin{prop}\label{prop:characterization_rate_of_convergence}
Let $\bfY$ be a stochastic process in $\barC$, $\bfV$ an SGPP with $D$-norm
$\norm\cdot_D$ and $\bfeta$ an SMSP with $P(\bfeta\le f)=\exp\left(-\norm f_D\right)$, $f\in E^-[0,1]$.
\begin{itemize}
  \item[\textit{(i)}] Suppose that $\bfY$ is in the spectral
      $\delta$-neighborhood of $\bfV$ for some $\delta\in(0,1]$. Then we
      have
      \[
        \sup_{f\in\barE}\abs{P\left(\bfY\le \frac f n\right)^n-P(\bfeta\le f)}=O\left(n^{-\delta}\right).
      \]
  \item[\textit{(ii)}] Suppose that $H_f(c)=P(\bfY\le c\abs f)$ is
      differentiable with respect to $c$ in a left neighborhood of $0$ for
      any $f\in E^-_1[0,1]$, i.e., $h_f(c):=(\partial/\partial c) H_f(c)$
      exists for $c\in(-\varepsilon,0)$ and any $f\in E^-_1[0,1]$. Suppose,
      moreover, that $H_f$ satisfies the von Mises condition
      \[
        \frac{-ch_f(c)}{1-H_f(c)}=:1+r_f(c)\to_{c\uparrow 0} 1,\qquad f\in E^-_1[0,1],
      \]
      with remainder term $r_f$ satisfying
      \[
        \sup_{f\in E^-_1[0,1]}\abs{\int_c^0\frac{r_f(t)} t\,dt}\to_{c\uparrow 0} 0.
      \]
      If
      \[
        \sup_{f\in\barE}\abs{P\left(\bfY\le \frac f n\right)^n-P(\bfeta\le f)}=O\left(n^{-\delta}\right)
      \]
      for some $\delta\in(0,1]$, then $\bfY$ is in the spectral
      $\delta$-neighborhood of the GPP $\bfV$.
 \end{itemize}
 \end{prop}

\begin{proof}
Note that
\begin{align*}
  &\sup_{f\in\barE}\abs{P\left(\bfY\le \frac f n\right)^n-P(\bfeta\le f)}\\
  &= \sup_{f\in\barE}\abs{P\left(\bfY\le \frac{\norm f_\infty} n \frac f{\norm f_\infty}\right)^n- P\left(\bfeta\le \norm f_\infty  \frac f{\norm f_\infty}\right)}\\
  &=\sup_{c < 0}\sup_{f\in E^-_1[0,1]}\abs{P\left(\bfY\le \frac cn\abs f\right)^n- P\left(\bfeta\le c\abs f\right)}\\
  &=\sup_{f\in E^-_1[0,1]} \sup_{c < 0}\abs{P\left(\bfY\le \frac cn\abs f\right)^n- P\left(\bfeta\le c\abs f\right)}.
\end{align*}
The assertion now follows by repeating the arguments in the proof of Theorem
1.1 in  \citet{falr02}, where the bivariate case has been established.
\end{proof}

\section{Distributional Differentiability of an SMSP}\label{sec:differentiability_smsp}

In this section, our aim is to establish the concept of \emph{distributional differentiability} of an SMSP $\bm\eta=(\eta_t)_{t\in[0,1]}$. In Proposition \ref{prop:differentiabiliy_of_smsp} we will prove the following result: Let $\bfZ=(Z_t)_{t\in[0,1]}\in C[0,1]$ be a generator process of $\bfeta$. Suppose that $Z'_t=(\partial/\partial t)Z_t$
exists for $t=t_0$ a.s. Then $(\eta_{t_0+h}-\eta_{t_0})/h$ converges in distribution to some rv on the real line as $h\to 0$ and we compute its df.

This is a first result on differentiability of max-stable processes. To the best of our knowledge, the question, under which conditions a max-stable process is differentiable at $t_0\in[0,1]$ a.s., is an open problem. If $\eta'_t=(\partial/\partial t)\eta_t$ actually exists at $t=t_0$ a.s., the distribution of $\eta'_{t_0}$ equals that of $-\eta_{t_0}\zeta_{t_0}$, where the rv $\zeta_{t_0}$ is independent of $\eta_{t_0}$ and has df $\mathbb F_{t_0}(x)=E\left(1_{\set{Z'_{t_0}\le x Z_{t_0}}}Z_{t_0}\right)$, $x\in\R$; see the discussion after Proposition \ref{prop:differentiabiliy_of_smsp}.

As an auxiliary result, we first compute the partial derivatives of a functional $D$-norm $\norm\cdot_D$. For this purpose, we need the following definition. Let $\mathcal X$ be a normed function space, and $J:\mathcal X\to\R$ a functional. The \emph{first variation} (or the \emph{G\^{a}teaux differential}) \emph{of $J$ at $u\in\mathcal X$ in the direction $v\in\mathcal X$} is defined as
\begin{equation*}
\nabla J(u)(v):=\lim_{\eps\to0}\frac{J(u+\eps v)-J(u)}{\eps}=\frac{d}{d\eps}J(u+\eps v)\Big|_{\eps=0}.
\end{equation*}
Moreover, the \emph{right-hand} (\emph{left-hand}) first variation of $J$ at $u$ in the direction $v$ is defined as
\begin{equation*}
\nabla^+ J(u)(v):=\lim_{\eps\downarrow0}\frac{J(u+\eps v)-J(u)}{\eps}\quad\text{and}\quad\nabla^- J(u)(v):=\lim_{\eps\downarrow0}\frac{J(u)-J(u-\eps v)}{\eps}.
\end{equation*}

Considering a $D$-norm $\norm\cdot_D$ a functional on the space $E[0,1]$, we can calculate the first variation of $\norm\cdot_D$.  The choice of the space $E[0,1]$ allows us the incorporation of the fidis and therefore yields the partial derivatives of a \emph{multivariate} $D$-norm. This finite-dimensional version of the following result has already been observed by \citet{einkraseg12}. Note that as a norm is a convex function, a multivariate $D$-norm $\norm{\bfx}_D$ is for almost every $\bfx<\bfzero\in\R^d$ continuously differentiable.

\begin{lemma}\label{lem:dnorm_firstvariation}
Let $\norm\cdot_D$ be a $D$-norm on the function space $E[0,1]$ with generator $\bm Z=(Z_t)_{t\in[0,1]}\in C[0,1]$. Let $t_0\in[0,1]$ and $1_{\{t_0\}}\in E[0,1]$ be the indicator function of the one point set $\{t_0\}$. Then for every $f\in E[0,1]$
\begin{align*}
\nabla^+\norm f_D\left(1_{\{t_0\}}\right)&=\lim_{\eps\downarrow0}\frac{\norm{f+\eps1_{\{t_0\}}}_D-\norm{f}_D}{\eps}\\
&=\begin{cases}-E\left(1_{\{\sup_{t\neq t_0}\abs{f(t)}Z_t<\abs{f(t_0)}Z_{t_0}\}}Z_{t_0}\right),\quad f(t_0)<0,\\ +E\left(1_{\{\sup_{t\neq t_0}\abs{f(t)}Z_t\leq\abs{f(t_0)}Z_{t_0}\}}Z_{t_0}\right),\quad f(t_0)\geq0,\end{cases}
\end{align*}
and
\begin{align*}
\nabla^-\norm f_D\left(1_{\{t_0\}}\right)&=\lim_{\eps\downarrow0}\frac{\norm{f}_D-\norm{f-\eps1_{\{t_0\}}}_D}{\eps}\\
&=\begin{cases}-E\left(1_{\{\sup_{t\neq t_0}\abs{f(t)}Z_t\leq\abs{f(t_0)}Z_{t_0}\}}Z_{t_0}\right),\quad f(t_0)\leq0,\\ +E\left(1_{\{\sup_{t\neq t_0}\abs{f(t)}Z_t<\abs{f(t_0)}Z_{t_0}\}}Z_{t_0}\right),\quad f(t_0)>0.\end{cases}
\end{align*}
\end{lemma}

\begin{proof}
Let $f\in E[0,1]$ with $f(t_0)>0$. We have
\begin{equation*}
\norm f_D=E\left(\sup_{t\in[0,1]}\abs{f(t)}Z_t\right).
\end{equation*}
First we calculate the right-hand first variation of $\norm\cdot_D$, assuming $f(t_0)\geq 0$. For $\eps>0$ there exists a disjoint decomposition of the underlying probability space $(\Omega,\mathcal A,P)$ via
\begin{equation*}
\Omega=A_1+A_2^{\eps}+A_3^{\eps}
\end{equation*}
with
\begin{align*}
A_1:&=\left\{\sup_{t\neq t_0}\abs{f(t)}Z_t\leq f(t_0)Z_{t_0}=\sup_{t\in[0,1]}\abs{f(t)}Z_t\right\},\\
A_2^{\eps}:&=\left\{f(t_0)Z_{t_0}<\sup_{t\neq t_0}\abs{f(t)}Z_t=\sup_{t\in[0,1]}\abs{f(t)}Z_t\leq (f(t_0)+\eps)Z_{t_0}\right\}\downarrow_{\eps\downarrow0}\emptyset\\
A_3^{\eps}:&=\left\{(f(t_0)+\eps)Z_{t_0}<\sup_{t\neq t_0}\abs{f(t)}Z_t=\sup_{t\in[0,1]}\abs{f(t)}Z_t\right\}.
\end{align*}
Therefore we obtain by the dominated convergence theorem
\begin{align*}
\nabla^+&\norm f_D\left(1_{\{t_0\}}\right)=\lim_{\eps\downarrow0}\frac{\norm{f+\eps1_{\{t_0\}}}_D-\norm{f}_D}{\eps}\\
=&\lim_{\eps\downarrow0}E\left(\frac1\eps\left(\max\left(\sup_{t\neq t_0}\abs{f(t)}Z_t,(f(t_0)+\eps)Z_{t_0}\right)-\sup_{t\in[0,1]}\abs{f(t)}Z_t\right)\right)\\
=&\lim_{\eps\downarrow0}E\left(\frac1\eps\Big((f(t_0)+\eps)Z_{t_0}-f(t_0)Z_{t_0}\Big)\cdot 1_{A_1}\right)\\
&+\lim_{\eps\downarrow0}E\left(\frac1\eps\Big((f(t_0)+\eps)Z_{t_0}-\sup_{t\neq t_0}\abs{f(t)}Z_t\Big)\cdot 1_{A_2^{\eps}}\right)\\
&+\lim_{\eps\downarrow0}E\left(\frac1\eps\Big(\sup_{t\neq t_0}\abs{f(t)}Z_t-\sup_{t\neq t_0}\abs{f(t)}Z_t\Big)\cdot 1_{A_3^{\eps}}\right)\\
=&E\left(Z_{t_0}\cdot1_{A_1}\right)
\end{align*}
since the second summand after the second to last equality vanishes as
\begin{align*}
\frac1\eps\Big((f(t_0)+\eps)Z_{t_0}-\sup_{t\neq t_0}\abs{f(t)}Z_t\Big)\cdot 1_{A_2^{\eps}}&<\frac1\eps\Big((f(t_0)+\eps)Z_{t_0}-f(t_0)Z_{t_0}\Big)\cdot 1_{A_2^{\eps}}\\
&\to_{\eps\downarrow0}Z_{t_0}\cdot 1_{\emptyset}=0.
\end{align*}
Note that
\begin{equation*}
\sup_{t\neq t_0}\abs{f(t)}Z_t\leq f(t_0)Z_{t_0}\iff \sup_{t\in[0,1]}\abs{f(t)}Z_t=f(t_0)Z_{t_0}.
\end{equation*}
The case $f(t_0)\leq0$ works similarly. In order to calculate the left-side first variation of $\norm\cdot_D$, assuming $f(t_0)>0$, we find for $\eps>0$ a disjoint decomposition of $(\Omega,\mathcal A,P)$ via
\begin{equation*}
\Omega=B_1^{\eps}+B_2^{\eps}+B_3
\end{equation*}
with
\begin{align*}
B_1^{\eps}:&=\left\{\sup_{t\neq t_0}\abs{f(t)}Z_t< (f(t_0)-\eps)Z_{t_0}\right\}\uparrow_{\eps\downarrow0}\left\{\sup_{t\neq t_0}\abs{f(t)}Z_t< f(t_0)Z_{t_0}\right\}=:B_1,\\
B_2^{\eps}:&=\left\{(f(t_0)-\eps)Z_{t_0}\leq\sup_{t\neq t_0}\abs{f(t)}Z_t<(f(t_0))Z_{t_0}\right\}\downarrow_{\eps\downarrow0},\emptyset\\
B_3:&=\left\{f(t_0)Z_{t_0}\leq\sup_{t\neq t_0}\abs{f(t)}Z_t\right\}.
\end{align*}
Hence we obtain again by the dominated convergence theorem
\begin{align*}
\nabla^-&\norm f_D\left(1_{\{t_0\}}\right)=\lim_{\eps\downarrow0}\frac{\norm{f}_D-\norm{f-\eps1_{\{t_0\}}}_D}{\eps}\\
=&\lim_{\eps\downarrow0}E\left(\frac1\eps\left(\sup_{t\in[0,1]}\abs{f(t)}Z_t-\max\left(\sup_{t\neq t_0}\abs{f(t)}Z_t,(f(t_0)-\eps)Z_{t_0}\right)\right)\right)\\
=&\lim_{\eps\downarrow0}E\left(\frac1\eps\Big(f(t_0)Z_{t_0}-(f(t_0)-\eps)Z_{t_0}\Big)\cdot 1_{B^{\eps}_1}\right)\\
&+\lim_{\eps\downarrow0}E\left(\frac1\eps\Big(f(t_0)Z_{t_0}-\sup_{t\neq t_0}\abs{f(t)}Z_t\Big)\cdot 1_{B_2^{\eps}}\right)\\
&+\lim_{\eps\downarrow0}E\left(\frac1\eps\Big(\sup_{t\neq t_0}\abs{f(t)}Z_t-\sup_{t\neq t_0}\abs{f(t)}Z_t\Big)\cdot 1_{B_3}\right)\\
=&E\left(Z_{t_0}\cdot1_{B_1}\right)
\end{align*}
since the second summand after the second last equality vanishes by the same argument as above. The case $f(t_0)<0$ works similarly.
\end{proof}

The first variation (or the partial derivatives, respectively) of a $D$-norm emerge in the easiest case of the so-called \emph{prediction problem}, cf. \citet{wansto11}, \citet{domeyri12} and \citet{domeyimin12}. Suppose one knows the distribution of an SMSP $\bm\eta$ and the point $\{\eta_{t_0}=x\}$, $x<0$, has already been observed. We are interested in the conditional distribution of $\bm\eta$, given $\{\eta_{t_0}=x\}$. The finite-dimensional version of the following Lemma is part of Proposition 4.2 in \citet{domeyimin12}.

\begin{lemma}\label{lem:conditional_singlepoint_functional}
Let $\bm\eta=(\eta_t)_{t\in[0,1]}$ be an SMSP with $D$-norm $\norm\cdot_D$ generated by $\bm Z=(Z_t)_{t\in[0,1]}$. Choose an aribtrary $t_0\in[0,1]$. Then for every $f\in E^-[0,1]$ with $f(t_0)=0$ and almost all $x<0$
\begin{equation*}
P\left(\bm\eta\leq f\big|\eta_{t_0}=x\right)=\exp\left(-\left(x+\norm{f+x1_{\{t_0\}}}_D\right)\right)\cdot E\left(1_{\{\sup_{t\in[0,1]}\abs{f(t)}Z_t\leq\abs{x}Z_{t_0}\}}Z_{t_0}\right).
\end{equation*}
\end{lemma}

\begin{proof}
The rv $\eta_{t_0}$ has Lebesgue-density $e^y$, $y\leq 0$. Therefore, we have
by basic rules of conditional distributions for almost all $x<0$
\begin{align*}
P\left(\bm\eta\leq f\big|\eta_{t_0}=x\right)&=\lim_{\eps\downarrow0}\frac{\eps^{-1}P(\bm\eta\leq f,\eta_{t_0}\in(x,x+\eps])}{\eps^{-1}P(\eta_{t_0}\in(x,x+\eps])}\\
&=\exp(-x)\lim_{\eps\downarrow0}\frac{P(\bm\eta\leq f,\eta_{t_0}\leq x+\eps)-P(\bm\eta\leq f,\eta_{t_0}\leq x)}{\eps}.
\end{align*}
Now define the function $g\in E^-[0,1]$ by $g(t)=f(t)$, $t\neq t_0$, and $g(t_0)= x$. Then we have by Lemma \ref{lem:dnorm_firstvariation}
\begin{align*}
P\left(\bm\eta\leq f\big|\eta_{t_0}=x\right)&=\exp(-x)\lim_{\eps\downarrow0}\frac{\exp\left(-\norm{g+\eps1_{\{t_0\}}}_D\right)-\exp\left(-\norm{g}_D\right)}{\eps}\\
&=-\exp(-x)\exp\left(-\norm{g}_D\right)\cdot\nabla^+\norm g_D\left(1_{\{t_0\}}\right)\\
&=\exp\left(-\left(x+\norm{f+x1_{\{t_0\}}}_D\right)\right)\cdot E\left(1_{\{\sup_{t\in[0,1]}\abs{g(t)}Z_t=\abs{x}Z_{t_0}\}}\right)\\
&=\exp\left(-\left(x+\norm{f+x1_{\{t_0\}}}_D\right)\right)\cdot E\left(1_{\{\sup_{t\in[0,1]}\abs{f(t)}Z_t\leq\abs{x}Z_{t_0}\}}\right).
\end{align*}
\end{proof}

The following lemma on the distribution of the increments of an SMSP can be shown by elementary calculations.

\begin{lemma}\label{lem:increments_smsp}
Consider an SMSP $\bm\eta=(\eta_t)_{t\in[0,1]}$ with generator process $\bm Z=(Z_t)_{t\in[0,1]}$ and choose arbitrary $s,t\in[0,1]$, $s\neq t$. Denote by $\norm\cdot_D$ the $D$-norm pertaining to $(\eta_s,\eta_t)$. Then for every $x\in\R$
\begin{align*}
&P\left(\eta_s-\eta_t\leq x\right)\\
&=\begin{cases}\displaystyle\int_{-\infty}^0\exp\left(-\norm{(x+y,y)}_D\right)\cdot E\left(1_{\{yZ_t\leq(x+y)Z_s\}}Z_t\right)~dy,&\; x<0 \\ \displaystyle\int_{-\infty}^{-x}\exp\left(-\norm{(x+y,y)}_D\right)\cdot E\left(1_{\{yZ_t\leq(x+y)Z_s\}}Z_t\right)~dy+1-\exp(-x),&\; x\geq0. \end{cases}
\end{align*}
\end{lemma}

Note that the only possible point of discontinuity of the df $P(\eta_s-\eta_t\leq x)$ is $x=0$ where $P(\eta_s-\eta_t\leq 0)=\left(\norm{(1,1)}_D\right)^{-1}E\left(1_{\{Z_t\geq Z_s\}}Z_t\right)$.

The preceding lemma allows us to introduce the following differentiability concept. Firstly, we call a stochastic process $(X_t)_{t\in[0,1]}$ \emph{almost surely differentiable in $t_0\in[0,1]$}, if the difference quotient $(X_{t_0+h}-X_{t_0})/h$ converges almost surely to some rv $X_{t_0}'$ on the real line for $h\to0$.

Now different to that, we call a stochastic process $(Y_t)_{t\in[0,1]}$ \emph{differentiable in distribution in $t_0\in[0,1]$}, if the difference quotient $(Y_{t_0+h}-Y_{t_0})/h$ converges in distribution to some rv on the real line for $h\to0$.

Lastly, we call a stochastic process $\bm\xi=(\xi_t)_{t\in[0,1]}$ pathwise differentiable on $[0,1]$ if every path $\bm\xi(\omega)$ is differentiable on $[0,1]$.

\begin{prop}[Differentiability in Distribution of SMSP]\label{prop:differentiabiliy_of_smsp}
Let $\bfeta=(\eta_t)_{t\in[0,1]}$ be an SMSP with generator process $\bfZ=(Z_t)_{t\in[0,1]}\in C[0,1]$. Suppose that for some $t_0\in[0,1]$
\begin{equation}\label{eqn:local_differentiability_of_generator_process}
\frac{Z_{t_0+h}-Z_{t_0}}{h} \to_{h\to 0} Z'_{t_0}\quad\textrm{a.s.}
\end{equation}
Then we have for $x\neq 0$
\[
P\left(\frac{\eta_{t_0+h}-\eta_{t_0}}{h}\le x\right)\to_{h\to 0} H_{t_0}(x):= \int_{-\infty}^0 \exp(y) E\left( 1_{\left\{Z'_{t_0}\le -\tfrac xy Z_{t_0}\right\}}Z_{t_0}\right)\,dy.
\]
\end{prop}

\begin{proof}
We have for $x\neq0$ and $h>0$ by Lemma \ref{lem:increments_smsp}
\begin{align*}
&P\left(\eta_{t_0+h}-\eta_{t_0}\le hx\right)\\
&= \int_{-\infty}^{-h\abs x} \exp\left(-\norm{(hx+y,y)}_{D(h)}\right)\cdot E\left(1_{\{yZ_{t_0}\leq(hx+y)Z_{t_0+h}\}}Z_{t_0}\right)\,dy + o(1)
\end{align*}
as $h\downarrow0$, where $\norm{\cdot}_{D(h)}$ is the $D$-norm generated by $(Z_{t_0+h},Z_{t_0})$. Now we obtain for almost all $y<-h\abs x$
\begin{align*}
&E\left( 1_{\left\{y Z_{t_0} \leq (hx+y) Z_{t_0+h}\right\}}Z_{t_0}\right)\\
&=E\left(1_{\left\{y \tfrac{Z_{t_0}-Z_{t_0+h}}{h} \leq x Z_{t_0+h}\right\}}Z_{t_0} \right)\\
&=E\left(1_{\left\{ \tfrac{Z_{t_0+h}-Z_{t_0}}{h} \leq -\tfrac xy Z_{t_0+h}\right\}}Z_{t_0} \right)\\
&\to_{h\downarrow 0} E\left( 1_{\left\{ Z'_{t_0} \le -\tfrac xy Z_{t_0}\right\}}Z_{t_0}\right)
\end{align*}
by condition \eqref{eqn:local_differentiability_of_generator_process} which implies the assertion if $h\downarrow0$. On the other hand, we have for $x\neq0$ and $h<0$ by Lemma \ref{lem:increments_smsp}, condition \eqref{eqn:local_differentiability_of_generator_process}, and the fact that $E(Z_{t_0})=1$
\begin{align*}
&P\left(\eta_{t_0+h}-\eta_{t_0}\ge hx\right)=1-P\left(\eta_{t_0+h}-\eta_{t_0}\leq hx\right)\\
&=1- \int_{-\infty}^{h\abs x} \exp\left(-\norm{(hx+y,y)}_{D(h)}\right)\cdot E\left(1_{\{yZ_{t_0}\leq(hx+y)Z_{t_0+h}\}}Z_{t_0}\right)\,dy + o(1)\\
&\to_{h\uparrow0}1-\int_{-\infty}^0 \exp(y) E\left( 1_{\left\{Z'_{t_0}\ge -\tfrac xy Z_{t_0}\right\}}Z_{t_0}\right)\,dy\\
&=1-\int_{-\infty}^0\exp(y)~dy+\int_{-\infty}^0 \exp(y) E\left( 1_{\left\{Z'_{t_0}\le -\tfrac xy Z_{t_0}\right\}}Z_{t_0}\right)\,dy\\
&=\int_{-\infty}^0 \exp(y) E\left( 1_{\left\{Z'_{t_0}\le -\tfrac xy Z_{t_0}\right\}}Z_{t_0}\right)\,dy.
\end{align*}
\end{proof}

Proposition \ref{prop:differentiabiliy_of_smsp} gives a sufficient condition on the differentiability in distribution of an SMSP. However, it does not imply differentiability of the path of $\bfeta$ at $t_0$. But if $\bfeta$ is differentiable at $t_0$ a.\,s., then $H_{t_0}$ is the df of the derivative $(\partial/\partial t)\eta_t$ of $\bfeta$ at $t=t_0$. We, therefore, denote by $\eta_{t_0}'$ a rv which follows the df $H_{t_0}$.

Suppose that $\bm Z$ is a.\,s. differentiable in $t_0$. Then
\[
\mathbb{F}_{t_0}(x):= E\left(1_{\left\{Z_{t_0}'\le x Z_{t_0}\right\}}Z_{t_0} \right),\qquad x\in\R,
\]
defines a common df on $\R$. Denote by $\zeta_{t_0}$ a rv which follows this df and which is independent of $\eta_{t_0}$. Then we obtain the equation
\[
H_{t_0}(x) =P\left(-\eta_{t_0} \zeta_{t_0}\le x\right),\qquad x\in\R,
\]
i.e., we have
\[
\eta_{t_0}'=_D -\eta_{t_0} \zeta_{t_0}.
\]
The pathwise derivative of $\bfeta$ at $t_0$, if it exists, coincides, therefore, in distribution with $-\eta_{t_0} \zeta_{t_0}$.

\begin{lemma}\label{lem:equality_of_expectations}
Suppose that $E\left(Z_{t_0}'\right)$ exists. Then the mean value of $\mathbb F_{t_0}$ exists as well and coincides with $E\left(Z_{t_0}'\right)$.
\end{lemma}

\begin{proof}
The expectation of an arbitrary rv $\xi$ exists iff $\int_0^\infty P(\xi>x)\,dx+\int_{-\infty}^0 P(\xi<x)\,dx<\infty$, and in this case
\[
E(\xi)= \int_0^\infty P(\xi>x)\,dx - \int_{-\infty}^0 P(\xi<x)\,dx.
\]

As a consequence we obtain from Fubini's theorem
\begin{align*}
&\int x\,\mathbb F_{t_0}(dx)\\
&= \int_0^\infty 1-E\left(1_{\left\{Z_{t_0}'\le x Z_{t_0}\right\}}Z_{t_0} \right)\,dx - \int_{-\infty}^0 E\left(1_{\left\{Z_{t_0}'\le x Z_{t_0}\right\}}Z_{t_0} \right)\,dx\\
&= \int_0^\infty E\left(1_{\left\{Z_{t_0}'> x Z_{t_0}\right\}}Z_{t_0} \right)\,dx - \int_{-\infty}^0 E\left(1_{\left\{Z_{t_0}'\le x Z_{t_0}\right\}}Z_{t_0} \right)\,dx\\
&= \int z \int_0^\infty 1_{\{z'>x z\}}\,dx\,\left(P*\left(Z_{t_0},Z_{t_0}'\right)\right)(d(z,z'))\\
&\hspace*{2cm} -\int z \int_0^\infty 1_{\{z'\le x z\}}\,dx\,\left(P*\left(Z_{t_0},Z_{t_0}'\right)\right)(d(z,z'))\\
&= \int z\max\left(\frac{z'} z,0\right)\left(P*\left(Z_{t_0},Z_{t_0}'\right)\right)(d(z,z'))\\
&\hspace*{2cm} +  \int z\min\left(\frac{z'} z,0\right)\left(P*\left(Z_{t_0},Z_{t_0}'\right)\right)(d(z,z'))\\
&= E\left(Z_{t_0}'\right).
\end{align*}
\end{proof}

As a consequence we obtain in particular
\[
E\left(\eta_{t_0}'\right) = - E\left(\eta_{t_0}\zeta_{t_0}\right) = - E\left(\eta_{t_0}\right) E\left(\zeta_{t_0}\right) = E\left(Z_{t_0}'\right).
\]

We close this section by giving some examples how Proposition \ref{prop:differentiabiliy_of_smsp} can be applied.

\begin{exam}\upshape
Put
\[
Z_t:= U\cos^2(\lambda t) + V \sin^2(\lambda t),\qquad t\in[0,1],
\]
where $U\ge 0$, $V\ge 0$ are rv with $E(U)=E(V)=1$ and $\lambda\in\R$. The process $\bfZ=(Z_t)_{t\in[0,1]}$ is pathwise differentiable with
\[
\frac\partial{\partial t}Z_t=\lambda\sin(2\lambda t)(V-U)=:Z'_t.
\]
The distribution of the derivative in distribution $\eta'_t$ is accessible under additional conditions on $U$ and $V$,  but it follows immediately from Lemma \ref{lem:equality_of_expectations} that in general $E(\eta'_t)=E(Z'_t)=0$.
\end{exam}

\begin{exam}\upshape
The constant generator process $Z_t\equiv1$, $t\in[0,1]$, gives rise to an SMSP $\bm\eta=(\eta_t)_{t\in[0,1]}$ with totally dependent univariate margins. The paths of this SMSP are constant a.\,s., which means that $\eta'_t=0$ a.\,s. This fact is reflected in Proposition \ref{prop:differentiabiliy_of_smsp}. If $Z_t\equiv1$, $t\in[0,1]$, then $\mathbb F_t(x)=1_{[0,\infty)}(x)$, which implies
\begin{equation*}
H_t(x)=\int_{-\infty}^0\exp(y)\mathbb F_t(-x/y)~dy=1_{[0,\infty)}(x).
\end{equation*}
\end{exam}

\begin{exam}\upshape
Let $(Z_0,Z_1)$ be the generator of a bivariate standard max-stable rv $(\eta_0,\eta_1)$ with independent margins. It is well-known that such a generator has to fulfill
\begin{equation*}
P(Z_0=0,Z_1=2)=P(Z_0=2,Z_1=0)=1/2.
\end{equation*}
Now define a generator process by
\begin{equation*}
Z_t:=Z_0+t(Z_1-Z_0)\left(=\max((1-t)Z_0,tZ_1)\right),\quad t\in[0,1],
\end{equation*}
and denote by $\bm\eta=(\eta_t)_{t\in[0,1]}$ the pertaining SMSP. Then obivously $\bm Z=(Z_t)_{t\in[0,1]}$ is pathwise differentiable with $Z'_t=Z_1-Z_0$, $t\in[0,1]$. Lemma \ref{lem:equality_of_expectations} instantly implies $E(\eta_t')=E(Z'_t)=0$.

Furthermore, we have for $x\in\R$ and $t\in[0,1]$
\begin{equation*}
\mathbb F_t(x)=E\left(1_{\{Z'_t\leq xZ_t\}}Z_t\right)=\begin{cases}1,&\qquad x\geq\frac1t,\\
1-t,&\qquad \frac1{t-1}\leq x<\frac1t,\\
0,&\qquad x<\frac1{t-1}.
\end{cases}
\end{equation*}
Hence, the corresponding rv $\zeta_t$ that follows the df $\mathbb F_t$ is discrete with $P(\zeta_t=\frac1{t-1})=1-t$ and $P(\zeta_t=\frac1t)=t$. Therefore we obtain
\begin{align*}
H_t(x)&=\int_{-\infty}^0\mathbb F_t(-x/y)\exp(y)\,dy\\
&=\begin{cases}\displaystyle\int_{-\infty}^{x(1-t)}\exp(y)(1-t)\,dy,&\qquad x<0, \\
\displaystyle\int_{-\infty}^{-xt}\exp(y)(1-t)\,dy+\int_{-xt}^{0}\exp(y)\,dy,&\qquad x\geq 0
\end{cases}\\
&=\begin{cases}(1-t)\exp(x(1-t)),&\qquad x<0, \\ 1-t\exp(-xt),&\qquad x\geq 0.
\end{cases}
\end{align*}
\end{exam}

\begin{exam}\upshape
Let $Z_0\sim\mathcal U(0,2)$ and $Z_1$ be a rv with $Z_0+Z_1=2$ a.\,s. Then $Z_1\sim\mathcal U(0,2)$ and $(Z_0,Z_1)$ defines a (bivariate) generator. Define the generator process
\begin{equation*}
Z_t:=Z_0+t(Z_1-Z_0),\quad t\in[0,1].
\end{equation*}
Again, $\bm Z=(Z_t)_{t\in[0,1]}$ is pathwise differentiable with $Z'_t=Z_1-Z_0$, $t\in[0,1]$. Clearly, $Z'_t\sim\mathcal U(-2,2)$, $t\in[0,1]$. This, along with the fact that $Z_{1/2}=1$ a.\,s., implies
\begin{equation*}
\mathbb F_{1/2}(x)=E\left(1_{\{Z'_{1/2}\leq x\}}\right)=P(Z'_{1/2}\leq x).
\end{equation*}
Hence, $\mathbb F_{1/2}$ is the df of the uniform distribution on $(-2,2)$. Elementary calculations now show that
\begin{equation*}
H_{1/2}(x)=\begin{cases}-\frac x4\Ei\left(\frac x2\right)+\frac12\exp\left(\frac x2\right), &\quad x<0,\\
-\frac x4\Ei\left(-\frac x2\right)-\frac12\exp\left(-\frac x2\right)+1, &\quad x>0,
\end{cases}
\end{equation*}
where $\Ei(x)=\int_{-\infty}^x\frac{\exp(t)}{t}~dt$ denotes the exponential integral which is well-defined for  $x<0$. Furthermore, we have $H_{1/2}(0)=1/2$. In particular, $H_{1/2}$ is continuous in 0 since the exponential integral satisfies $x\Ei(x)\to_{x\to0}0$. The density of $H_{1/2}$ is given by $h_{1/2}(x)=-\Ei\left(-\abs x/2\right)/4$, $x\not=0$.
\end{exam}

\end{document}